\documentclass[11pt]{amsart}

\usepackage{mathdots}

\usepackage{psfrag}
\usepackage{amsmath}
\usepackage{amssymb}
\usepackage[english]{babel}
\usepackage{amsthm}
\usepackage{amsfonts}
\usepackage[applemac]{inputenc}
\usepackage{tikz}
\usetikzlibrary{arrows}
\usetikzlibrary{patterns}
\usetikzlibrary{shapes}
\usepackage{mathrsfs}
\usepackage{bbm}

\newtheorem{teo}{Theorem}
\newtheorem{lemma}[teo]{Lemma}

\newtheorem{coro}[teo]{Corollary}
\newtheorem{propo}[teo]{Proposition}

\newcommand{\R}{\mathbb{R}}
\newcommand{\ii}{\mathrm{i}}

\newcommand{\eps}{\epsilon}
\newcommand{\N}{\mathbb{N}}
\newcommand{\C}{\mathbb{C}}
\newcommand{\CP}{\mathbb{C}\textrm{P}}
\newcommand{\RP}{\mathbb{R}\mathrm{P}}

\newcommand{\Z}{\mathbb{Z}}

\newcommand{\Sym}{\textrm{Sym}(n,\R)}

\newcommand{\Q}{\mathcal{Q}}
\newcommand{\PP}{\mathbb{P}}
\newcommand{\E}{\mathbb{E}}

\theoremstyle{remark} 
\newtheorem{remark}[]{Remark}

\title{Random matrices and the average topology of the intersection of two quadrics}

\author{ A. Lerario}
\thanks{Department of Mathematics, Purdue University}

\begin{document}

\maketitle

\begin{abstract}
Let $X_\R$ be the zero locus in $\RP^n$ of one or two independently and Weyl distributed random real quadratic forms.  Denoting by $X_\C$ the complex part in $\CP^n$ of $X_\R$ and by $b(X_\R)$ and $b(X_\C)$ the sums of their Betti numbers, we prove that: \begin{equation}\label{ex}\lim_{n\to\infty}\frac{\mathbb{E}b(X_\R)}{n}=1.\end{equation}
In particular for one quadric hypersurface asymptotically Smith's inequality $b(X_\R)\leq b(X_\C)$ is expected to be sharp. The methods we use combine Random Matrix Theory, Integral Geometry and spectral sequences.

\end{abstract}
\section{Introduction}
Let us consider the real vector space $W_{n,d}$ of real homogeneous polynomials of degree $d$ and $n+1$ variables. Each $f\in W_{n,d}$ defines a complex algebraic set $X_{\C}$ in $\CP^n$ and for an open dense subset of $W_{n,d}$ all these algebraic sets have the same volume (induced from the Fubiny-Study one) and the same topology. The first statement directly follows from Wirtinger's formula and the second essentially from the fact that the set of degenerate polynomials has real codimension two in the space of polynomials with complex coefficients.  If we look at the zero locus $X_\R$ of $f$ in the real projective space $\RP^n$, then the situation dramatically changes. The set of degenerate polynomials has now real codimension one and as we cross it the topology of $X_\R$ (and its volume) may change. It is however possible to compare the volume and the topology of the real and complex parts by mean of the following inequalities:
\begin{equation}\label{volsmith}\frac{\textrm{Vol}(X_\R)}{\textrm{Vol}(\RP^{n-1})}\leq \frac{\textrm{Vol}(X_\C)}{{\textrm{Vol}(\CP^{n-1})}}\quad \textrm{and}\quad b(X_\R)\leq b(X_\C).\end{equation}
The l.h.s. inequality directly follows from the integral geometry formula and the r.h.s. is the so called Smith's inequality and involves the sum of the Betti numbers (in this paper all cohomology groups and related ranks are assumed to be with $\Z_2$ coefficients).\\
This raises the question: if $f$ is picked up randomly, what do we expect the volume and the topology of its real zero locus to be? Clearly this question does not make sense for the complex part, since for any reasonable distribution of probability on $W_{n,d}$ the volume and the sum of the Betti numbers of $X_\C$ are constant functions outside of a zero probability set.\\
To answer this question let us consider the example of a random polynomial $f$ in $W_{1,d}$. In this case both the volume and the sum of the Betti numbers of $X_{\R}$ equal the number of real (projective) roots of $f$. In the seminal paper \cite{Kac} Kac proved that if the coefficients of $f$ are distributed as standard independent gaussians with mean zero and variance one, then the expected number of real roots $E_d$ of $f$ satisfies:
$$\lim_{d\to \infty} \frac{E_d}{\log d}=\frac{2}{\pi}.$$
In this paper we will assume $f=\sum f_\alpha x^\alpha$, where $x^\alpha=x_0^{\alpha_0}\cdots x_n^{\alpha_n}$ and the $f_\alpha$ are Gaussian independently distributed random variables with mean zero and variance $\frac{d! }{ \alpha_0!\cdots \alpha_n!}$. The resulting distribution of probability on $W_{n,1}$ is called the \emph{Weyl} distribution (or the Kostlan distribution). For instance the expected number of zeroes $E_d$ of a random Weyl distributed polynomial in $W_{1,d}$ is given exactly by:
$$E_d=\sqrt{d}.$$
The reader is referred to the paper \cite{EdKo} for a proof of both these limits and a survey of related results.\\
More generally the expected volume of the random algebraic variety $X_\R$ defined by a set of polynomials $f_1, \ldots, f_k$ with each $f_i\in W_{n,d_i}$ Weyl and independently distributed is given by:
$$\mathbb{E}[\textrm{Vol}(X_\R)]=\sqrt{d_1\cdots d_k}\textrm{Vol}(\RP^{n-k}).$$
Indeed the previous formula was proved in a sequence of papers of Shub and Smale first \cite{ShSm} and  Bürgisser \cite{Bu} in this general form. In this last paper the formula follows from the more striking result on the computation of the expected curvature polynomial of $X_\R$ in $\RP^n$. The same computation also gives a precise formula for the expected Euler characteristic of $X_\R$ (the hypersurface case was already done by Podkorytov in \cite{Po}). In the case of $n$ equations in $\RP^n$ this expected volume gives the expected number of solutions of a random polynomial system.\\
Concerning the sum of the Betti numbers of $X_\R$, very little is known. Even the case of the expected number of components of a random real, Weyl distributed curve of degree $d$ in $\RP^2$ is not known. Gayet and Welschinger in \cite{GaWe1} proved that maximal curves, i.e. those with approximatively $d^2$ components, become exponentially rare in the degree. The same authors in \cite{GaWe2} proved that the expected total Betti number of a random 
Weyl distributed hypersurface of degree $d$ in $\RP^n$ satisfies the following:
$$\lim_{d\to \infty}\mathbb{E}\bigg[\frac{b(X_\R)}{d^n}\bigg]=0.$$
In an unpublished letter \cite{Sarnak} Sarnak claims that in the case of a plane curve we have even $\lim_{d\to\infty}\mathbb{E}\big[\frac{b(X)}{d}\big]\leq c_1$, for a positive constant $c_1.$ Indeed in a different direction Nazarov and Sodin \cite{NaSo} proved that the expected number of connected components of a random\footnote{The distribution of probability here is such that the components of the spherical harmonic with respect to the $L^2_{S^2}$ orthonormal basis are i.i.n. distributed.} spherical harmonic of degree $d$ is asymptotically $c_2d^2$, for some $c_2>0.$ Generalizing this result, the author together with E. Lundberg, was able to prove that the expectation of the number of connected components of a random\footnote{Here the probability distribution is a real analogue of the Weyl one, suggested by P. Sarnak in \cite{Sarnak}.} hypersurface of degree $d$ in $\RP^n$ is asymptotically of order $d^n$ (see \cite{LeLu}).\\
In this paper we study the case the random algebraic set is the intersection of real quadrics in $\RP^n$. In this case Barvinok's bound (see \cite{Ba}) gives for the intersection $X_\R$ of $k$ quadrics in $\RP^n$:
$$b(X_\R)\leq n^{O(k)}.$$  
This bound suggests that the measure of the complexity of $X_\R$ is the number $k$ of quadrics we are intersecting. Motivated by this and Smith's inequality (\ref{volsmith}) we thus  focus on a different asymptotics, namely we fix the number of equations, i.e. the codimension of $X_\R$, and we let the number of variables go to infinity. The case we study is somehow the simplest, i.e. the one when $X_\R$ is defined by one or two random Weyl independent quadratic equations, but offers some new perspectives. More specifically we prove that if $X_\R$ is the intersection of one or two independently and Weyl distributed quadrics then:
\begin{equation}\label{exp}\lim_{n\to\infty}\frac{ \mathbb{E} b(X_\R)}{n}=1.\end{equation}
Thus as we increase the number of variables, Smith's inequality for one quadric hypersurface is expected to be sharp.\\
The key fact here is that given a quadratic form $q$ on $\R^n$ we can associate to it a symmetric matrix $Q$ of order $n$ (using a scalar product) and the form $q$ is Weyl distributed if and only if $Q$ is in the \emph{Gaussian Orthogonal ensemble}. This simple observation allows to introduce the language of Random Matrix Theory into the problem. For the case of one quadric hypersurface it is then enough to study the expectation of the signature of $Q$, which characterizes the topology of the zero locus of $q$.\\ 
For the case of the intersection of two quadric hypersurfaces, the idea for proving these limits is to relate the sum of the Betti numbers of $X_\R$ to that of its spectral variety, namely the intersection in the space of all quadratic forms of the linear system defining $X_\R$ with the set of singular quadrics. This is made rigorous by the introduction of a spectral sequence from \cite{AgLe} to compute the cohomology of the intersection of real quadrics. This kind of duality between the variables and the quadratic equations is the same that allows to prove Barvinok's bound.\\
In the case of the intersection of three random quadrics in $\RP^n$, the spectral variety is a random curve, but its distribution of probability is fairly different from the Weyl or the standard one. This random curve is smooth with probability one and its topological complexity is essentially the topological complexity of $X_\R$ (see \cite{Le3}).\\
The paper is organized as follows: in Section 2 we introduce some notation and review some notions from integral geometry and in Section 3 we discuss the technique from \cite{AgLe} to study the intersection of real quadrics, focusing on the case of one and two quadrics. In Section 4 we prove the limit (\ref{exp}) for one quadric; this is obtained by a combination of a formula for the cohomology of one single quadric and Wigner's semicircular law. In Section 5 we consider the case of two quadrics: here the result follows again from a formula for the cohomology derived from Section 3. This formula involves the number of singular quadrics in the linear system defining $X_\R$ and the maximum of the inertia index of the quadrics belonging to this linear system; both the expectation of these numbers are computed using the integral geometry formula. As a byproduct we compute the intrinsic volume in the Frobenius norm of the set $\Sigma$ of singular symmetric matrices of norm one; this computation is related to some limit of gap probabilities in the GOE and the theory of Painlevé equations. Finally in the Appendix we compute the expected value of the rank of the second differential of the spectral sequence from Section 3.

\section*{Acknowledgements}
The author is grateful to Sofia Cazzaniga, who implemented numerical simulations to verify the obtained results, and to Erik Lundberg for useful discussions.

\section{Random quadratic forms and integral geometry}

Let $q=\sum c_{ij}x_ix_j$ be a real quadratic form whose coefficients $c_{ij}$ are independent Gaussian random variables with mean zero and variance one for $i=j$ and two for $i\neq j$. The quadratic form $q$ is said to be a \emph{Weyl distributed} random polynomial. This results in a distribution of probability on the space $\Q(n)$ of real quadratic forms in $n$ variables; this distribution of probability is invariant by the action (by change of variables) of the orthogonal group $O(n).$ If $q$ is a random quadratic form Weyl distributed as above and $\R^{m}$ is a linear subspace of $\R^{n}$, then the restriction $q|_{\R^{m}}$ is again a random quadratic form Weyl distributed (see \cite{Bu}). Equivalently, once a scalar product has been fixed, it is possible to associate to each quadratic form $q$ a symmetric matrix $Q$ by the equation:
$$q(x)=\langle x, Qx\rangle,\quad \textrm{for all } x \in \R^{n}.$$
In this way a linear isomorphism between the space $\Q(n)$ of real homogeneous polynomials of degree two in $n$ variables and the space $\Sym$ of real symmetric matrices of order $n$ is set up; we denote by $N=\frac{1}{2}n(n+1)$ the dimension of this vector spaces. If $q$ is a random quadratic form Weyl distributed, the corresponding random matrix $Q$ is said to belong to the \emph{Gaussian Orthogonal Ensemble}. The entries of $Q$ are independent Gaussian random variables with mean zero and variance one on the diagonal and one-half off diagonal. If we define the norm of a matrix $Q$ in $\Sym$ by $$\|Q\|^2=\textrm{tr}(Q^2),$$ then the induced probability distribution is uniform on the unit sphere $S^{N-1}$; this distribution is thus invariant by the action of the group $O(N)$ of orthogonal transformation of $\Sym$. The map $\rho:O(n)\to O(N)$ given by
$$\rho(M)Q=MQM^{-1}$$
defines a homomorphism of group: in fact if $M$ is orthogonal, then $\textrm{tr}(MQM^{-1})=\textrm{tr}(Q);$ the induced action of $O(n)$ on $\Q(n)$ is by change of variables. Thus we see that on $\Sym$ there are two actions, one of $O(N)$ and one of $O(n)$, both by isometries.\\
If $X$ is a compact riemannian manifold of dimension $d$ we denote its Riemannian density by $\omega_X$ and we define its normalized volume $p(X)$ to be the number $\frac{\textrm{Vol}(X)}{\textrm{Vol}{(S^d)}}$. An important tool to study the volume of submanifolds of the sphere, or the projective space, is the so called \emph{integral geometry formula}. Let $A$ and $B$ be submanifolds, with or without boundaries, of the unit sphere $S^m$, of dimension respectively $a$ and $b$, with $a+b\geq m$. We endow the sphere $S^m$ with the standard Riemannian metric and $A$ and $B$ by the induced one. The integral geometry formula is:
$$\frac{\int_{SO(m+1)} p(A\cap gB) dg}{\textrm{Vol}(SO(m+1))}= p(A)p(B).$$
where the integral is with respect to the Haar measure. A similar formula holds in the case $A$ and $B$ are submanifolds of the projective space $\RP^m$; in this case the volumes are normalized by $\textrm{Vol}(\RP^m).$

\section{The cohomology of the intersection of real quadrics}
We recall in this section a general construction to study the topology of the intersection of real quadrics. If we are given quadratic forms $q_{1},\ldots, q_{k}$ on $\R^{n}$, then we can consider their common zero locus $X_\R$ in $\RP^{n-1}:$
$$X_\R=X_{\R}(q_{1},\ldots, q_{k})$$
To study the topology of $X_\R$ we consider the linear span $W$ of $\{q_1, \ldots, q_k\}$ in the vector space $\Q(n)$:
$$W=\textrm{span}\{q_1, \ldots, q_k\}$$ 
The arrangement of $W$ with respect to the subset of degenerate quadratic forms (those with at least one dimensional kernel) determines the topology of the base locus in the following way. The simplest invariant we can associate to a quadratic form $q$ is its positive inertia index $\ii^{+}(q)$, namely the maximal dimension of a subspace $V\subset \R^{n}$ such that $q|_{V}$ is positive definite. In a similar fashion we consider for $j\in \N$ the sets:
	$$\Omega^{j}=\{q\in W\backslash \{0\}\,|\, \ii^{+}(q)\geq j\}.$$ In this way we get a filtration $\Omega^{n}\subseteq\Omega^{n-1}\subseteq\cdots\subseteq \Omega^{1}\subseteq\Omega^{0}$ of $W\backslash\{0\}$ by open sets.
The following theorem is proved in \cite{AgLe}.
\begin{teo}\label{AgLe}There exists a cohomology spectral sequence of the first quadrant $(E_r, d_r)_{r\geq1}$ converging to $H^{n-1-*}(X_\R)$ such that
$$E_2^{i,j}=H^i(W, \Omega^{j+1}).$$
\end{teo}
Notice that each set $\Omega^{j+1}$ deformation retracts to an open subset of the unit sphere in $W$; because of this deformation retraction in the sequel we will always think at each set $\Omega^{j+1}$ as a subset of the unit sphere. From the previous theorem we immediately derive the following inequality:
$$b(X_\R)\leq n+\sum_{j\geq 1} b(\Omega^j)$$ In the low codimension cases, i.e. for $k=1,2,3$, the previous formula can be sharpened as following. The spectral variety of the linear system $W$ is defined to be the intersection of the set $\Sigma$ of degenerate forms of norms one (i.e. of symmetric matrices with zero determinant and Frobenius norm one) with $W$:
$$\Sigma_W=W \cap \Sigma.$$ Notice that by homogeneity of the determinant this definition does actually not depend on the norm on $\Q(n)$ (respectively $\textrm{Sym}(n,\R))$. For a generic choice of $q_1, \ldots, q_k$ the following properties are satisfied:  the vector space $W$ has dimension $k$; the intersection of $W$ with the set of degenerate quadratic forms $\Sigma$ is transversal to every of its strata (the stratification is given by the dimension of the kernel). Thus, generically, for $k=1$, i.e. one quadric, the spectral variety is empty, for $k=2$ consists of a finite number of points and for $k=3$ it is a smooth curve (this follows from the fact that the codimension of the singular locus of $\Sigma$ is at least three). These are the only cases in which we may assume the spectral variety is generically smooth. If we define the number $$\mu_W=\max \ii^+|_W,$$ then in the case $k=1$ generically we have \begin{equation}\label{one}b(X_\R)=2(n-\mu_W).\end{equation}
The reader is referred to \cite{AgLe}, Example 2 for this formula,  while for $k=2,3$ the following inequality holds: $b(X_\R)\leq 3n-4\mu_W+\frac{1}{2}b(\Sigma_W)$.

  \section{The case of one quadric hypersurface}
In this section we study the expected total Betti number of the zero locus $X_\R$ of one single quadric Weyl distributed. We start by recalling some results from random matrix theory. Let $Q$ be a random matrix in the Gaussian Orthogonal Ensemble (recall that this is equivalent to the corresponding quadratic form $q$ being Weyl distributed).
If $\lambda_1,\ldots, \lambda_n$ are the eigenvalues of $Q$, we define the \emph{empirical spectral distribution}
$$\tau_n=\frac{1}{n}\sum_{i=1}^{n}\delta_{\lambda_i/\sqrt{n}}.$$
Strictly speaking $\tau_n$ is a random variable in the space of the probability distributions over $\R$; what is relevant for us is that once we have a continuous, compactly supported function $\psi$ we can define the random variable $X_n(\psi)=\int_\R\psi d\tau_n$. Wigner's semicircular law concerns the limit of the expectation of such a random variable. Let $\tau_{sc}$ be the probability density on the real line
$$\tau_{sc}=\frac{1}{2\pi}(4- |x|^2)_+^{1/2}dx,$$
and for every $\psi$ continuous and compactly supported let $X_{sc}(\psi)$ be the number $\int_\R\psi d\tau_{sc}$. The following theorem was proved by Wigner (see \cite{W} for the original work and \cite{TT} for a modern exposition).
\begin{teo}[Wigner]\label{sc}
For every interval $A\subset \R$ 
$$\lim_{n\to\infty}\mathbb{E}\int_Ad\tau_n=\int_Ad\tau_{sc}.$$
Moreover for every $\psi$ continuos and compactly supported and $F$ continuous bounded 
$$\lim_{n\to\infty}\mathbb{E}F(X_n(\psi))=F(X_{sc}(\psi)).$$

\end{teo}

Using the previous theorem we can prove the following proposition.

\begin{propo}
Let $Q$ be a random symmetric matrix in the Gaussian Orthogonal Ensemble of dimension $n$. Let also $\mu_{n}=\max\{\ii^+, \ii^-\}$ and  $\nu_n=\min\{\ii^+,\ii^-\}.$ Then:
$$\lim_{n\to \infty}\frac{\mathbb{E}[\mu_n-\nu_n]}{n}=0.$$
\begin{proof}
First notice that $\mu_n(Q)=\mu_n(Q/\sqrt{n})$ and $\nu_n(Q)=\nu_n(Q/\sqrt{n})$, since the inertia index of a symmetric matrix is invariant by multiplication of a positive number. Thus we have the equality of random variables:
$$\frac{\mu_n-\nu_n}{n}=\frac{|\ii^+-\ii^-|}{n}=\bigg|\int_\R Hd\tau_n\bigg|,$$
where $H(x)=\textrm{sign}(x)$; the first equality follows directly from the definition and the second comes from: $$\int_\R Hd\tau_n=\int_{(0,\infty)} d\tau_n-\int_{(-\infty,0)}d\tau_n=\frac{\ii^+-\ii^-}{n},$$ (notice that since the set of symmetric matrices with determinant zero is  the complement of a full measure set, we can discard the term $\int_{\{0\}}d\tau_n$). For every $\eps>0$ let us now consider a continuous, compactly supported function $\psi_\eps$ satisfying: $\psi_\eps$ is odd, $|\psi_\eps|\leq 1$ and $\psi_{\eps}(x)=H(x)$ for $x\in A(\eps)=(-3, \eps)\cup(\eps, 3)$. The existence of such a function is obvious. Let also $F$ be any compactly supported function equal to $|x|$ for $|x|\leq 1$. We have now the following chain of inequalities of random variables:

\begin{align*}
\bigg|\int_\R H d\tau_n\bigg| &\leq F(X_n(\psi_\eps))+\bigg|\int_\R H-\psi_\eps d\tau_n\bigg|
\\  &\leq F(X_n(\psi_\eps))+\int_{\R\backslash A(\eps)}d\tau_n.
\end{align*}
The first inequality comes from the fact that $|X_n(\psi_\eps)|\leq 1$ and the definition of $F$; the second inequality is because $H-\psi_\eps$ is zero on $A(\eps)$ and $|H-\psi_\eps|\leq 1.$ Thus by the previous theorem \ref{sc} we have: 
$$\lim_{n\to\infty}\mathbb{E}F(X_n(\psi_\eps))=F(X(\psi_\eps))=0,$$ since $\psi_\eps$ is odd, and $$\lim_{n\to \infty}\mathbb{E}\int_{\R\backslash A(\eps)} d\tau_n=\int_{\R\backslash A(\eps)} d\tau_{sc}\leq 2\eps.$$ Hence for every $\eps>0$
$$\lim_{n\to\infty}\frac{\mathbb{E}[\mu_n-\nu_n]}{n}\leq 2\eps,$$
which together with $\mu_n-\nu_n\geq 0$ proves the proposition.
\end{proof}
\end{propo}
We derive the following theorem for the expected value of the total Betti number of a random quadratic hypersurface in $\RP^{n-1}.$

\begin{teo}\label{one}
Let $q$ be a random, Weyl distributed, quadratic form on $\R^{n}$ and $X_\R$ be its zero locus in $\RP^{n-1}.$ Then
$$\lim_{n\to \infty}\frac{\mathbb{E}[b(X_\R)]}{n}=1.$$
\begin{proof}
Since generically $\mu_{n}(q)=n-\nu_{n}(q)$, by theorem \ref{AgLe} we have $b(X_\R)=n+\nu_{n}(q)-\mu_{n}(q)$ (this is a restatement of formula \ref{one}). Since if $q$ is Weyl distributed then the corresponding symmetric matrix is in the Gaussian Orthogonal Ensemble, the conclusion follows from the limit of the previous proposition.
\end{proof}
\end{teo}

\begin{remark}
If we notice that for a nonsingular real quadric in $\CP^n$ we have $b(X_\C)=n+\frac{1}{2}(1+(-1)^{n+1})$, then the previous limit can be rewritten in a more fashionable way as:
$$\lim_{n\to\infty}\mathbb{E}[b(X_\R)/b(X_\C)]=1$$
The fact that this limit had to be less or equal then one is the content of Smith's inequalities (see the Appendix of \cite{Wilson}).
\end{remark}

\section{The case of the intersection of two quadrics}

In the case $X_\R$ is the intersection of two quadrics $q_1, q_2$ in $\RP^{n-1}$, we can derive directly from Theorem \ref{AgLe} the following. Recall the definition of $W$ as $\textrm{span}\{q_1, q_2\}\subset \Q(n)$, the number $\mu_{W}=\max \ii^+|_{W}$ and the spectral variety $\Sigma_W=\Sigma\cap W$ (in this case it is a subvariety of $S^1$, i,e consists either of a finite number of points or is the whole $S^1$). The topology of the intersection of two quadrics was studied by the author in \cite{Le2}; in fact the following proposition follows directly from Theorem 8 of \cite{Le2}, though we give a short proof here using Theorem 1.

\begin{propo} For a generic pair $(q_1, q_2)$ we have 
$$b(X_\R)= 3n-1-4\mu_W +(c_W+d_W) +\frac{1}{2}b(\Sigma_W).$$
where $c_W$ and $d_W$ belong to $\{0,1\}$.
\begin{proof}
In this case, summing  the elements for the third (the last) page of the spectral sequence of theorem \ref{AgLe} gives:
$$b(X_\R)=\textrm{rk}(E_3)=n-1-2(\mu_W -\min \ii^+|_W) +c_W+d_W+\sum_{j=\min \ii^+|_W}^{\mu_W-1}b_0(\Omega^{j+1}).$$
where we have called $c_W$ and $d_W$ respectively $\textrm{rk}(E_3^{0, \mu})$ and $\textrm{rk}(E_3^{2, \mu-1})$; a direct look at the second table of the spectral sequence gives $c_W,d_W\in \{0,1\}$.
Now for a generic choice of $q_1, q_2$ we have $\min \ii^+|_W=n+1-\mu_W$, $\Sigma_W$ consists of a finite number of points and the function $\ii^+$ jumps exactly by $\pm 1$ when crosses $\Sigma_W.$ In particular each point of $\Sigma_W$ belongs exactly to one of the $\partial \Omega^{j+1}$, $\min \ii^+|_W\leq j\leq \mu_W-1$. Thus Alexander-Pontryiagin duality gives:
$$\sum_{j=\min \ii^+|_W}^{\mu_{W}-1}b_0(\Omega^{j+1})=\frac{1}{2}b(\Sigma_W).$$
which yelds the desired formula.

\end{proof}
\end{propo}
In particular we see that
\begin{equation}\label{extwo}\mathbb{E}[b(X_\R)]=3n-1-4\mathbb{E}[\mu_W]+\mathbb{E}[c_W+d_W]+\frac{1}{2}\mathbb{E}b(\Sigma_W).\end{equation}
We will compute each term of the previous sum; we start by introducing some auxiliary material. We recall now that the probability that the interval $(-\eps, \eps)$ does not contain any of the eigenvalues of a $Q\in \textrm{GOE}(n)$ is called \emph{gap probability}; we consider this probability as a function of $\epsilon$ and denote it by $f_n(\eps).$
In the case $n$ is even, $f_n(\eps)$ can be evaluated using methods from integrable systems. Following \cite{FoWhi} we have:\footnote{Here we use the same notation as in \cite{FoWhi} to help the reader comparing with this reference. The subscript of $\sigma_V$ is due to the connection with the Painlevé fifth equation.}
\begin{equation}\label{f}f_n(\eps)=\tau_{\sigma_V}(\eps^2),\end{equation}
where $\tau_{\sigma_V}$ is a function satisfying:
\begin{equation}\label{tau}\sigma_V(t)=t\frac{d}{dt}\log \tau_{\sigma_V}(t)\quad \textrm{and}\quad \lim_{t\to 0^+}\sigma_{V}(t)t^{-1/2}=-\frac{\Gamma(\frac{n+1}{2})}{\Gamma(\frac{n}{2})\Gamma(\frac{1}{2})\Gamma(\frac{3}{2})}=-c_n. \end{equation}
We denote by $\Gamma$ the Euler Gamma Function; $\sigma_{V}$ itself satisfies a second order differential equation (equation (2.9) of \cite{FoWhi}) but for our purposes is not necessary to write it down explicitly.
\begin{lemma}\label{bound2}For every $n\in 2\mathbb{N}$:
$$\lim_{\eps\to 0^+}f_n'(\eps)=-2c_n.$$
Moreover as $n$ goes to infinity $c_n\sim \frac{\sqrt{2n}}{\pi}.$
\end{lemma}
\begin{proof}
For the first limit we argue as follows. Since $\sigma_V(t)=t\tau_{\sigma_V}'(t)/\tau_{\sigma_V}(t)$ we have $\tau_{\sigma_V}'(t)=\sigma_V(t)\tau_{\sigma_V}(t)/t$ and thus:
$$\lim_{\eps\to 0^+}f_n'(\eps)=\lim_{\eps\to 0^+}2\eps \tau_{\sigma_V}'(\eps^2)=\lim_{\eps\to 0^+}\frac{2\sigma_V(\eps^2)\tau(\eps^2)}{\eps}=-2c_n,$$
where for the last limit we have used the limit in equation (\ref{tau}) and the fact that $\lim_\eps f_n(\eps)=\lim_{\eps}\tau_{\sigma_V}(\eps^2)=1$ by definition.\\
The second limit follows immediately using the known values of the Gamma function $\Gamma(1/2)=\sqrt{\pi}$, $\Gamma(3/2)=\sqrt{\pi}/2$ and the Stirling's asymptotic.

\end{proof}
We consider now the unit sphere $S^{N-1}$ in $\textrm{Sym}(n, \R)$ (with respect to the Frobenius norm) and the set $\Sigma$ of singular matrices of norm one. We are interested in bounding the intrinsic volume of $\Sigma$. In order to do that we start by bounding the volume of an $\eps$-tube $\Sigma_\eps$ of $\Sigma$ in $S^{N-1}$ (the volume is computed with respect to the Riemannian structure induced on $S^{N-1}$ by the norm). Let:
$$\sigma(Q)=\min_{\lambda \in s(Q)}|\lambda|;$$
notice that using this notation the gap probability $f_n(\eps)$ equals $\PP\{\sigma(Q)\geq \epsilon\}.$
\begin{lemma}\label{bound3}
$$\emph{\textrm{Vol}}(\Sigma_\eps)\leq\emph{\textrm{Vol}}(S^{N-1})\cdot (1-\mathbb{P}\{\sigma(Q)\geq \eps \|Q\|\}).$$
\end{lemma}
\begin{proof}

We let $Z\subset \textrm{Mat}(n, \R)$ be the set of degenerate matrices and consider the following chain of inequalities for a matrix $Q\in \textrm{Sym}(n, \R)$:
$$d_{\textrm{Mat}(n,\R)}(Q, Z)\leq d_{\textrm{Sym}(n,\R)}(Q, Z\cap \textrm{Sym}(n,\R))\leq d_{S^{N-1}}(Q, \Sigma).$$
In particular the set $\Sigma_\epsilon$ is contained in the set $\{d_{\textrm{Mat}(n,\R)}(Q,Z)\leq \epsilon\}\cap S^{N-1},$ and by Eckart-Young theorem the last one equals $\{\sigma(Q)\leq \eps \|Q\|\}\cap S^{N-1};$ thus we get:
\begin{equation}\label{eq:bound1}
\textrm{Vol}(\Sigma_\epsilon)\leq\textrm{Vol}\{\sigma(Q)\leq \eps \|Q\|\}\cap S^{N-1}.
\end{equation}
Since the probability distribution of the Gaussian Orthogonal Ensemble is uniform on the sphere $S^{N-1},$ then the conclusion follows.
\end{proof}

\begin{remark}The probability $\mathbb{P}\{\sigma(Q)\geq \eps \|Q\|\}$ \emph{is not} the gap probability, because of the rescaling factor $\|Q\|$.
\end{remark}
The following lemma gives an upper bound for the volume of $\Sigma$.
\begin{lemma}\label{bound4}If $n\in 2\mathbb{N}$:
$$\emph{\textrm{Vol}}(\Sigma)\leq \emph{\textrm{Vol}}(S^{N-1})\cdot O(n c_n)$$
\end{lemma}
\begin{proof}

Let us start by defining the functions $g_n(\eps)=\PP\{\sigma(Q)\geq \eps\|Q\|\}$ and $\hat{\sigma}=\sigma|_{S^{N-1}}.$ Notice that in polar coordinates $(\theta, r)\in S^{N-1}\times (0, \infty)$ we have:
$$\{\sigma(Q)\geq \eps \|Q\|\}=\{\hat{\sigma}(\theta)\geq \eps\}\quad \textrm{and}\quad \{\sigma(Q)\geq \eps\}=\{\hat{\sigma}(\theta)\geq \eps/r\}.$$
Both identities follow by: $\sigma(Q)=\|Q\|\sigma(Q/\|Q\|)=r\hat{\sigma}(\theta).$
In particular we can write:
\begin{align}\nonumber
f_n(\eps)&=\frac{1}{(2 \pi)^{N/2}}\int_{S^{N-1}}\int_{0}^{\infty}\chi_{\{\hat{\sigma}(\theta)\geq \eps/r\}}(\theta, r)r^{N-1}e^{-\frac{r^2}{2}}dr dS^{N-1}=\\
&\nonumber=\frac{1}{(2 \pi)^{N/2}}\int_{0}^{\infty}\underbrace{\left(\int_{S^{N-1}}\chi_{\{\hat{\sigma}(\theta)\geq \eps/r\}}(\theta, r)dS^{N-1}\right)}_{\textrm{Vol}(S^{N-1})g_n(\eps/r)}r^{N-1}e^{-\frac{r^2}{2}}dr=\\
&=\frac{\textrm{Vol}(S^{N-1})}{(2 \pi)^{N/2}}\int_{0}^{\infty}g_n(\eps/r)r^{N-1}e^{-\frac{r^2}{2}}dr\nonumber
\end{align}
In particular from the last equation we get:
\begin{equation}\label{derivative}
f'_n(\eps)=\frac{\textrm{Vol}(S^{N-1})}{(2 \pi)^{N/2}}\int_{0}^{\infty}g'_n(\eps/r)r^{N-2}e^{-\frac{r^2}{2}}dr.
\end{equation}
The function $g_n$ is monotone decreasing (as we let $\eps$ increase the probability of $\{\sigma(Q)\geq \eps\|Q\|\}$ decreases); thus $g'_n\leq 0$
and by Fatou's Lemma from (\ref{derivative}) we get the following (notice reversed inequalities due to the sign of $g_n'$):
\begin{align}\nonumber\frac{\textrm{Vol}(S^{N-1})}{(2 \pi)^{N/2}}\int_{0}^{\infty}\lim_{\eps \to 0}g'_n(\eps)r^{N-2}e^{-\frac{r^2}{2}}dr&=\frac{\textrm{Vol}(S^{N-1})}{(2 \pi)^{N/2}}\int_{0}^{\infty}\lim_{\eps \to 0}g'_n(\eps/r)r^{N-2}e^{-\frac{r^2}{2}}dr\\ \nonumber
&\geq\lim_{\eps\to 0}\frac{\textrm{Vol}(S^{N-1})}{(2 \pi)^{N/2}}\int_{0}^{\infty}g'_n(\eps/r)r^{N-2}e^{-\frac{r^2}{2}}dr\\
&=\lim_{\eps\to 0}f_n'(\eps)=-2c_n
 \end{align}
 In particular we have obtained:
 $$\underbrace{\left(\frac{\textrm{Vol}(S^{N-1})}{(2 \pi)^{N/2}}\int_{0}^{\infty}r^{N-2}e^{-\frac{r^2}{2}}dr\right)}_{\frac{\Gamma(\frac{N-1}{2})}{\sqrt{\pi}\Gamma{(\frac{N}{2}})}} \cdot \lim_{\eps\to 0}-g'_n(\eps)\leq 2c_n $$
 which can be rewritten as:
 \begin{equation}\label{finaleq}\lim_{\eps\to 0}-g'_n(\eps)\leq 2c_n \frac{\sqrt{\pi}\Gamma(\frac{N}{2})}{\Gamma{(\frac{N-1}{2}})}=O(nc_n).\end{equation}
 We finally turn to the definition of $\textrm{Vol}(\Sigma)=\lim_{\eps\to 0}\textrm{Vol}(\Sigma_\eps)/2\eps$; plugging the result of Lemma \ref{bound3} into this limit we get:
 $$\textrm{Vol}(\Sigma)\leq \textrm{Vol}(S^{N-1})\cdot \lim_{\eps\to 0}\frac{1-g_n(\eps)}{2\eps}= \textrm{Vol}(S^{N-1})\cdot \lim_{\eps\to 0}-\frac{g'_n(\eps)}{2}\leq \textrm{Vol}(S^{N-1})\cdot O(nc_n)$$
 where in the second equality we have used De l'Hopital's rule and in the last one we used (\ref{finaleq}). This concludes the proof.
\end{proof}

As a corollary we get the following Theorem.

\begin{teo}\label{sqrt}If $n\in 2\mathbb{N}:$
$$\E b(\Sigma_W)\leq O(\sqrt{n})$$

\end{teo}
\begin{proof}
We start by noticing that by assumption for every $g\in SO(N)$ the random quadratic forms $q$ and $gq$ have the same distribution (here the action is not by change of variable, but directly on the space of the coefficients). Thus we have:
$$\mathbb{E}b(\Sigma_W)=\frac{\int_{SO(N)}\mathbb{E}[b(\Sigma_{gW})]dg}{\textrm{Vol(SO(N))}}=\frac{\mathbb{E}\int_{SO(N)}b(\Sigma_{gW})dg}{\textrm{Vol}(SO(N))}=\frac{2\textrm{Vol}(\Sigma)}{\textrm{Vol}(S^{N-2})}.$$
The first equality is because for every $g\in SO(N)$ we have $\mathbb{E}b(\Sigma_{gW})=\mathbb{E}b(\Sigma_{W})$; the second is just linearity of expectation, and the third one is the integral geometry formula (there is no expected value because the integral is constant).\\
Using now the bound given by Lemma \ref{bound4} we get:
$$\E b(\Sigma_W)\leq \frac{\textrm{Vol}(S^{N-1})}{\textrm{Vol}(S^{N-2})}O(nc_n).$$
Recalling the formula for the volume of the sphere $\textrm{Vol}(S^{k-1})=\frac{2\pi^{k/2}}{\Gamma(k/2)},$ we see that $$\frac{\textrm{Vol}(S^{N-1})}{\textrm{Vol}(S^{N-2})}\sim\frac{\sqrt{2}}{n \pi}.$$
This, together with the asymptotic $c_n\sim \frac{\sqrt{2n}}{ \pi}$ from Lemma \ref{bound2}, concludes the proof.
\end{proof}

\begin{propo}\label{p1}
$$\lim_{n\to \infty}\frac{\mathbb{E}[4\mu_{n}]}{n}=2.$$

\begin{proof}
We start by noticing that for an open dense set of $(Q_1, Q_2)$ the following inequalities hold:
$$\ii^+(Q_1)\leq\mu_{W}\leq \ii^+(Q_1)+b(\Sigma_W).$$
In fact for a generic pencil the index function ``jumps'' exactly by $\pm1$ when crosses $\Sigma_W$ and thus the maximum that can reach over $W$ is $\mu(Q_1)+b(\Sigma_W).$ Dividing by $n$ and taking expectations, Theorem \ref{sqrt} gives the result for $n\in 2\mathbb{N}.$
To prove that the statement holds also for odd $n$ we notice that restricting a Weyl distributed random quadratic form $q$ to a subspace $V\subset \R^{n}$ gives again a Weyl distributed random quadratic form $q|_V$ on $V\simeq \mathbb{R}^{\textrm{dim}(V)}$; since $\ii^+(q|_V)\leq \ii^+(q)$ we have:
$$\mathbb{E}[\mu_{n-1}]\leq \mathbb{E}[\mu_n]\leq \mathbb{E}[\mu_{n+1}].$$
This proves that the same limit holds for odd $n.$

\end{proof}
\end{propo}

As a corollary we prove the following theorem for the asymptotic of $\mathbb{E}[b(X_\R)]$.

\begin{teo}
Let $X_\R\subset \RP^n$ be the intersection of two random quadrics independent and Weyl distributed. Then
$$\lim_{n\to \infty}\frac{\mathbb{E}[b(X_\R)]}{n}=1, \quad \textrm{$n$ odd}.$$

\begin{proof}
The limit follows from formula (\ref{extwo}), Theorem \ref{sqrt} and the previous proposition, after noticing that $\mathbb{E}[c_W+d_W]\leq 2$.\\
\end{proof}
\end{teo}

\begin{remark}Notice in particular that since the total Betti number of the complete intersection of two quadrics in $\CP^n$ is $2n-2$, then in this case the expectation of Smith's inequality is turned into an equality for large $n$ (up to a factor $\frac{1}{2}$).
\end{remark}

\section*{Appendix: the expected second differential}
It is interesting now to compute also the expected value of the number $c_W+d_W$. By definition we have:
$$c_W=\textrm{rk}(E_3^{0, \mu})\quad \textrm{and}\quad d_W=\textrm{rk}(E_3^{2, \mu-1})$$
where $(E_r, d_r)_{r\geq 0}$ is the spectral sequence of theorem \ref{AgLe} and $\mu=\mu_W=\max \ii^+|_W.$ We recall now from \cite{AgLe} the definition of the second differential of this spectral sequence.  Consider the bundle $L_{\mu}\to \Omega^{\mu}$ whose fiber at the point $q\in \Omega^{\mu}$ is the positive eigenspace of $Q$ and whose vector bundle structure is given by its inclusion in $\Omega^{\mu}\times \R^{n+1}.$ We let $w_{1,\mu}\in H^{1}(\Omega^{\mu})$ be the first Stiefel-Whitney class of this bundle. From Theorem B of \cite{AgLe} it follows that for every $x\in E_2^{0,\mu}$ we have:
$$d_2^{0, \mu}(x)=(x\smile w_{1, \mu})|_{(W, \Omega^j)}.$$
In particular, since $E_3^{0, \mu}=\ker d_2^{0,\mu}$ and $E_3^{2, \mu-1}=H^1(\Omega^\mu)/ \textrm{Im} d_2^{0, \mu}$ we immediately get:
\begin{equation}\label{whit}c_W+d_W=1+b_1(\Omega^\mu) -2 \omega_{1, \mu},\end{equation}
where $\omega_{1,\mu}$ is $\textrm{rk}(d_2^{0, \mu})$ (thus $\omega_{1, \mu}$ ``is" the Stiefel-Whitney class $w_{1, \mu}$ thought as an element of $H^1(\Omega^\mu)\subset \Z_2$). Using this description we prove the following.

\begin{propo}For two Weyl, independent random quadrics in $\Q(n)$ we have
$$\mathbb{E}[c_W+d_W]=1+(-1)^{[\frac{n}{2}]}\mathbb{P}\bigg\{\ii^+|_{W\backslash \{0\}}\equiv \bigg[ \frac{n}{2} \bigg]\bigg\}.$$
\begin{proof}
In the case $n$ is odd, for a generic pair of quadrics $(q_1, q_2)$ the group $H^1(\Omega^\mu)$ has to be zero: this is because any generic linear family of quadrics in an odd number of variables contains at least a line of degenerate quadrics and thus the index function cannot be constant on the nonzero elements of the family. Thus $w_{1,\mu}=0$ and equation \ref{whit} gives the desired conclusion in this case.\\
In the case $n$ is even we use the following fact (see Proposition 2 of \cite{Agrachev1}): for a generic pair of symmetric matrices $(Q_1, Q_2)$ there exists an invertible matrix $M$ such that both $M^TQ_1M$ and $M^TQ_2M$ have the same block-diagonal shape with blocks of dimensions one or two. In particular the index function for the family $x_1Q_1+x_2Q_2$ is the sum of the index functions for the families $x_1B_1^k+x_2B_2^k$ (because the number of positive eigenvalues of a symmetric matrix is invariant by congruence), where $M^TQ_iM=\textrm{diag}(B_i^1, \ldots, B_i^m)$. Let us focus on the term $b_1(\Omega^\mu)$ in equation (\ref{whit}). Notice that
$$\mathbb{E}[b_1(\Omega^\mu)]=\mathbb{P}\bigg\{\ii^+|_{W\backslash \{0\}}\equiv \bigg[ \frac{n}{2} \bigg]\bigg\}.$$ This is because the only case in which $b_1(\Omega^\mu)$ is nonzero, for a generic pair, is when the index function is constant on the nonzero elements of $W$, and for a generic pair this constant has to be $\frac{n}{2}.$ On the other hand using the previous observation, we se that the only way for the index function to be constant on $W\backslash\{0\}$, for a generic pair, is when each block has dimension two and the index function for each block is constantly equal to one. It is a well-known result that the bundle of positive eigenspace for a two dimensional family of quadrics in two variables equals the Moebius bundle (see \cite{Le2}), hence for every block the corresponding Stiefel-Whitney class is nonzero. Thus it follows that for a generic pair $(Q_1, Q_2)$, if the index function is constant on $W\backslash\{0\}$, then it must be equal $\frac{n}{2}$ and by the Whitney product formula in this case:
$$w_{1, \mu}\equiv\frac{n}{2}\textrm{ mod }2.$$
Thus $b_1(\Omega^\mu)$ equals $1$ with probability $p_1=\mathbb{P}\bigg\{\ii^+|_{W\backslash \{0\}}\equiv \bigg[ \frac{n}{2} \bigg]\bigg\}$ and zero otherwise (both for the even and the odd case); when $n$ is even, $\omega_{1,\mu}$ equals $\frac{n}{2}$ modulo 2 i.e. $\frac{1}{2}(1+(-1)^{[\frac{n}{2}]+1})$ with probability $p_1$, and zero otherwise. Using equation (\ref{whit}) and the definition of expectation we immediately get the conclusion.

\end{proof}\end{propo}

In the case $W$ is two dimensional, by Theorem \ref{sqrt} we have:
\begin{equation}\label{p11}\lim_{n\to \infty}\mathbb{P}\bigg\{\ii^+|_{W\backslash \{0\}}\equiv \bigg[\frac{n}{2}\bigg]\bigg\}=0.\end{equation}
Hence $\mathbb{E}[c_W+d_W]\to 1;$ notice also that the previous statement also gives a probabilistic statement on the second differential of the spectral sequence of Theorem \ref{AgLe}; in fact using the limit (\ref{p11}) we immediately derive the following.

\begin{coro} For the intersection of two independent, Weyl, random quadrics in $\RP^n$ we have: 
$$\lim_{n\to \infty}\E[ \emph{\textrm{rk}}(d_2)]=0.$$

\end{coro}


\begin{thebibliography}{15}
\bibitem{Agrachev1} A. A. Agrachev,  R. V. Gamkrelidze: \emph{Quadratic maps and smooth vector valued functions; Euler Characteristics ol level sets}, Itogi nauki.
    VINITI. Sovremennye problemy matematiki. Novejshie dostigeniya, 1989, v.35, 179--234 

\bibitem{AgLe} A. A. Agrachev, A. Lerario: \emph{Systems of quadratic inequalities}, 	Proceedings of the London Mathematical Society, (2012) 105 (3).
\bibitem{Ba} A. I. Barvinok: \emph{On the Betti numbers of
semialgebraic sets defined by few quadratic inequalities}, Discrete and Computational Geometry , 22:1-18 (1999). 
\bibitem{BCR} J. Bochnak, M. Coste, M-F. Roy: \emph{Real Algebraic Geometry}, Springer-Verlag, 1998. 

\bibitem{Bu}P. Bürgisser, \emph{Average Euler characteristic of random algebraic varieties}, C. R. Acad. Sci. Paris, Ser. I 345 (2007) 507–512.
\bibitem{EdKo} E. Edelman, A. Kostlan, \emph{How many zeros of a random polynomial are real?}, Bulletin of the American Mathematical Society 32 (1995), 1-37.
\bibitem{FoWhi} P. J. Forrester, N. S. White, \emph{$\tau$-function evaluation of gap probabilities in orthogonal and symplectic matrix ensembles}, Nonlinearity 15, 2002, Pages 937-954
\bibitem{GaWe1}D. Gayet, J-Y. Welschinger, \emph{Exponential rarefaction of real curves with many components}, Publ. Math. Inst. Hautes Études Sci. No. 113 (2011), 69–96.
\bibitem{GaWe2}D. Gayet, J-Y. Welschinger, \emph{What is the total Betti number of a random real hypersurface?}, arXiv:1107.2288v1 
\bibitem{He} U. Helmke, \emph{Critical Points of Matrix Least Squares Distance Functions}, Linear Algebra and its Applications, 215, 15 January 1995, Pages 1-19
\bibitem{Kac} M. Kac, \emph{On the average number of real roots of a random algebraic equation},  Bull. Amer. Math. Soc. Volume 49, Number 4 (1943), 314-320.
\bibitem{Le2}A. Lerario, \emph{Convex Pencils of real quadratic forms}, Discrete and computational Geometry, Volume 48, Number 4 (2012), 1025-1047.
\bibitem{Le3} A. Lerario, \emph{The total Betti number of the intersection of three real quadrics}, Advances in Geometry, to appear.
\bibitem{LeLu} A. Lerario, E. Lundberg, \emph{Statistics on Hilbert's Sixteenth Problem}, arXiv:1212.3823.
\bibitem{NaSo} F. Nazarov, M. Sodin, \emph{On the Number of Nodal Domains of Random Spherical Harmonics}, 	arXiv:0706.2409v1 
\bibitem{Po} S. S. Podkorytov, \emph{The mean value of the Euler characteristic of an algebraic hypersurface}, Algebra i Analiz, 11(5):185–
193, 1999. English translation: St. Petersburg Math. J. 11(5) (2000), pp. 853–860.
\bibitem{Reid} M. Reid: \emph{The complete intersection of two or more quadrics},  1972.

\bibitem{Sarnak} P. Sarnak, \emph{Letter to B. Gross and J. Harris on ovals of random plane curves}, publications.ias.edu/sarnak/paper/510 (2011)
\bibitem{ShSm} M. Shub, S. Smale, \emph{Complexity of Bezout's teorem II: volumes and probabilities}, The collected papers of Stephen Smale, Volume 3 (pp 1402-1420)
\bibitem{TT}T. Tao, \emph{Topics in Random Matrix Theory}, Graduate Studies in Mathematics
2012
\bibitem{W}E. Wigner \emph{On the Distribution of the Roots of Certain Symmetric Matrices} Ann. of Math. 67, 325-328, 1958.
\bibitem{Wilson}G. Wilson, \emph{Hilbert's sixteenth problem}, Topology 17 (1978), 53-74.

\end{thebibliography}
\end{document}